\documentclass{article}
\usepackage{indentfirst,overpic}
\usepackage{amsthm}
\usepackage{amsfonts}
\usepackage{amssymb}

\newtheorem{remark}{Remark}

\newcommand{\COLORON}{1}
\newcommand{\NOTESON}{0}
\newcommand{\Debug}{0}

\newcommand{\nted}{nested}
\newcommand{\smA}{\ensuremath{\BS^n \sm A}}
\newcommand{\smZ}{\ensuremath{\BS^n \sm Z}}
\newcommand{\pfl}{piecewise flat}

\usepackage[usenames]{color} 
\usepackage{amsthm,amssymb,amsmath,bbm,enumerate,epsf,stmaryrd,accents}
\usepackage{graphicx} 
\usepackage{bmpsize}

\usepackage[bookmarks, colorlinks=false, breaklinks=true]{hyperref} 

\usepackage{authblk}

\newcommand{\comment}[1]{}
\newcommand{\COMMENT}[1]{}

\definecolor{darkgray}{rgb}{0.3,0.3,0.3}
\newcommand{\defi}[1]{{\color{darkgray}\emph{#1}}}

\comment{
	\begin{lemma}\label{}	
\end{lemma}
\begin{proof}

\end{proof}

\begin{theorem}\label{}
\end{theorem} 
\begin{proof} 	

\end{proof}

}

\newtheorem{proposition}{Proposition}[section]

\newtheorem{theorem}[proposition]{Theorem}
\newtheorem{corollary}[proposition]{Corollary}

\newtheorem{lemma}[proposition]{Lemma}

\newtheorem{conjecture}{{Conjecture}}[section]

\newtheorem{problem}[conjecture]{{Problem}}

\newtheorem{question}[conjecture]{{Question}}

\newtheorem{examp}[proposition]{Example}

\newcommand{\FIG}{0}

\ifnum \NOTESON = 1 \newcommand{\note}[1]{ 

\hspace*{-30pt}
	{\color{blue}  NOTE: \color{Turquoise}{\small  \tt \begin{minipage}[c]{1.1\textwidth}  #1 \end{minipage} \ignorespacesafterend }} 
	
	}
\else \newcommand{\note}[1]{} \fi

\newcommand{\afsubm}[1]{ \ifnum \Debug = 1 {\mymargin{#1}}
\fi} 

\ifnum \Debug = 1 
\else  \fi

\ifnum \FIG = 1 \newcommand{\fig}[1]{Figure ``{#1}''}
\else \newcommand{\fig}[1]{Figure~\ref{#1}} \fi

\ifnum \FIG = 1 
\else  \fi

\ifnum \Debug = 1 \usepackage[notref,notcite]{showkeys}
\fi

\ifnum \COLORON = 0 \renewcommand{\color}[1]{}
\fi

\newcommand{\N}{\ensuremath{\mathbb N}}
\newcommand{\R}{\ensuremath{\mathbb R}}

\newcommand{\Z}{\ensuremath{\mathbb Z}}

\newcommand{\BS}{\ensuremath{\mathbb S}}

\newcommand{\cb}{\ensuremath{\mathcal B}}
\newcommand{\cc}{\ensuremath{\mathcal C}}

\newcommand{\cu}{\ensuremath{\mathcal U}}

\newcommand{\sm}{\backslash}

\newcommand{\cls}[1]{\ensuremath{\overline{#1}}}

\newcommand{\ocirc}[1]{\ensuremath{\accentset{\circ}{#1}}}

\makeatletter
\DeclareRobustCommand{\cev}[1]{%
  \mathpalette\do@cev{#1}%
}
\newcommand{\do@cev}[2]{%
  \fix@cev{#1}{+}%
  \reflectbox{$\m@th#1\vec{\reflectbox{$\fix@cev{#1}{-}\m@th#1#2\fix@cev{#1}{+}$}}$}%
  \fix@cev{#1}{-}%
}
\newcommand{\fix@cev}[2]{%
  \ifx#1\displaystyle
    \mkern#23mu
  \else
    \ifx#1\textstyle
      \mkern#23mu
    \else
      \ifx#1\scriptstyle
        \mkern#22mu
      \else
        \mkern#22mu
      \fi
    \fi
  \fi
}

\makeatother

\newcommand{\nin}{\ensuremath{{n\in\N}}}

\newcommand{\Cg}{Cayley graph}

\newcommand{\Lr}[1]{Lemma~\ref{#1}}

\newcommand{\Tr}[1]{Theorem~\ref{#1}}

\newcommand{\Sr}[1]{Section~\ref{#1}}

\newcommand{\Prr}[1]{Pro\-position~\ref{#1}}

\renewcommand{\iff}{if and only if}
\newcommand{\fe}{for every}

\newcommand{\st}{such that}

\newcommand{\ti}{there is}

\newcommand{\sico}{simply connected}

\newcommand{\pl}{piecewise linear}

\newcommand{\CtS}{Cantor 3-sphere}

\newcommand{\labtequ}[2]{
 \begin{equation} \label{#1} 	\begin{minipage}[c]{0.9\textwidth}  #2 \end{minipage} \ignorespacesafterend \end{equation} }

\newcommand{\mymargin}[1]{
 \ifnum \Debug = 1
  \marginpar{%
    \begin{minipage}{\marginparwidth}\small%
      \begin{flushleft}%
        {\color{blue}#1}%
      \end{flushleft}%
   \end{minipage}%
  }%
 \fi
}%

\newcommand{\mySection}[2]{}

\begin{document}
	\title{Every countable compact subset of $\mathbb{S}^n$ is tame}
	
\author{Agelos Georgakopoulos\thanks{Supported by  EPSRC grants EP/V048821/1 and EP/V009044/1.}}
\affil{  {Mathematics Institute}\\
 {University of Warwick}\\
  {CV4 7AL, UK}}

\date{\today}

\maketitle

\begin{abstract}
We prove that any two countable, compact, subsets of $\BS^n, n\geq 2$ that are homeomorphic also have homeomorphic complements. Thus any wild subspace like the classical construction of Antoine must contain a Cantor set.
\end{abstract}


\section{Introduction}

Just over a century ago, Antoine \cite{Antoine,AntoineT} 
found a homeomorph $A$ of the Cantor set inside the 3-sphere $\BS^3$ the complement $\BS^3 \sm A$ of which is not homeomorphic to the complement of a \defi{standard} Cantor set in $\BS^3$, i.e.\ one contained in a linear arc inside $\BS^3$. Antoine's discovery triggered a wealth of constructions of wild Cantor sets in $\BS^n$ with complements having various properties \cite{BeCoWil,CaWrSli,DavEmb,DeGOsb,GaReZeUnc,GaReZeRig,Kirkor,Sher,SkoCan}. In particular, Blankinship \cite{Blankinship} generalised Antoine's construction to all dimensions $n\geq 3$. Antoine's work also led Alexander to the discovery of his horned sphere \cite{Alexander}, and there are interesting connections between such spheres and wild Cantor sets \cite{DavWil}.

The aim of this note is to show that such constructions are not possible, in any dimension, if we replace the Cantor set by any topological space that does not contain it as a subspace. Our main result is

\begin{theorem} \label{main thm}
Let $A,Z$ be two countable, compact, subsets of $\BS^n, n\geq 2$, and let $h: A \to Z$ be a homeomorphism. Then there is an automorphism $h'$ of $\mathbb{S}^n$ extending $h$. 

Moreover, $h'$ can be chosen to be orientation reversing (or preserving).
\end{theorem}

It is well-known that every uncountable closed subspace of a Polish ---i.e.\ completely metrizable and separable--- space contains a perfect set, and therefore a homeomorph of the Cantor set \cite{Kechris}. Combined with \Tr{main thm}, and the aforementioned constructions of \cite{Antoine,Blankinship}, this means that the Cantor set  is the smallest space $X$ \st\ that there are two homeomorphs of $X$ in $\BS^n, n\geq 2$ with non-homeomorphic complements.

Thus the countability condition  of \Tr{main thm} is indispensable. To see that the compactness conditions is also necessary, consider $A,Z\subset \mathbb{S}^n$ that are both homeomorphic to $\mathbb{Q}$, so that $A$ is dense in $\mathbb{S}^n$ while $B$ is not\footnote{I thank Wojtek Wawrow for this example.}. 

For $n=2$, the first statement of \Tr{main thm} can be deduced from a classical result of Richards'  \cite{Richards}, which applies more generally to any totally disconnected, compact, subspace of $\mathbb{S}^2$. In particular, Richards' result implies that every Cantor set $C \subset \mathbb{S}^2$ is \defi{tame}, i.e.\ $\mathbb{S}^2 \sm C$ is homeomorphic to the complement of a standard Cantor set in  $\mathbb{S}^2$.

\medskip

As an example application of \Tr{main thm}, we can use it to answer the following question that became popular in on-line forums \cite{SE_bijection_Z2,MO_ext_hom}:
\begin{question} \label{Q Z}
Does every bijection of $\Z^2$
extend to a homeomorphism of $\R^2$?
\end{question}
It may look surprising at first sight that e.g.\ the bijection of $\Z^2$ that coincides with the identity in the upper half-plane and coincides with the bijection $(n,m)\mapsto (-n,m)$ in the lower half-plane extends to a homeomorphism $h$ of $\R^2$. This may look more surprising if one takes into account that $h$ would have to map the standard \Cg\ $G$ of  $\Z^2$ ---i.e.\ the square lattice--- onto a quite distorted image $h(G)$. But \Tr{main thm}, applied with $A=Z$ being the 1-point compactification of $\Z^2$ says that it is always possible, and in any dimension. Moreover, the second statement says that it is possible to extend the identity on $\Z^2$ to an orientation reversing homeomorphism $h$ of $\R^2$. Again, it is difficult to imagine the image $h(G)$ of the square lattice: its vertices are fixed, but for every vertex $v\in \Z^2$, the clockwise cyclic ordering of its edges is reversed. (Conversely, given such an embedding $h(G)$ on the square lattice, one can deduce the existence of the desired automorphism of $\R^2$ by applying the Jordan--Schoenfliess theorem to each of the faces of $h(G)$; see \cite{ThomassenRichter} for a more general statement.)

\medskip

Proving \Tr{main thm} would be easy if we could find bases $\cb^A,\cb^Z$ for the topologies of $A,Z$, ideally induced by open metric balls of $\BS^n$, \st\ $h$ maps each element of $\cb^A$ to one of $\cb^Z$. In \Sr{sec nted} we define a notion of  \defi{\nted} basis, which is possible to find in every  countable, compact, metric space (\Lr{lem nted}), that will allow us to reduce to the aforementioned ideal situation;  $h$ will map each element of $\cb^A$ to a finite disjoint union of elements of $\cb^Z$, which union is always possible to engulf inside a single topological ball. The rest of the argument is an adaptation of an idea of Richards that ping-pongs between closed submanifolds of $A,Z$ bordered by topological spheres.


We  prove \Tr{main thm} in the more general setup where countability is relaxed to the existence of a \nted\ basis (\Tr{main thm unct}). As every tame Cantor set $C\subset \BS^n$ has a  \nted\ basis, we recover the well-known fact that $C$ is strongly homogeneously embedded, i.e. every automorphism of $C$ extends to an automorphism of $\BS^n$. 


\comment{
	\begin{corollary} \label{cor C}
Let $C$ be a tame Cantor set in $\BS^n$. Then there is an automorphism of $\BS^n$ that fixes $C$ pointwise and reverses the orientation.
\end{corollary}
}

\section{Preliminaries and proof ingredients} \label{sec nted}

Let $X$ be a topological space, and \cu\ a set of open subsets of $X$. We say that \cu\ is \defi{\nted}, if \fe\ $U,V\in \cu$ we have either $\cls{U} \cap \cls{V} = \emptyset$, or $U\subseteq V$, or $V\subseteq U$.

An \defi{open metric ball}  of a metric space  $(Y,d)$ is a set of the form $B_r(x):= \{y \in Y \mid d(x,y)<r \}$.

\begin{lemma} \label{lem nted}
Let $(Y,d)$ be a metric space, and $X\subseteq Y$ a countable 
subspace. Then \ti\ a nested set of open metric balls of $Y$ that form a basis of the topology of $X$ when restricted to it. 
\end{lemma}
\begin{proof}
Given some constant $r_0\in \R_{>0}$, we can produce a \nted\ cover $\cu= \cu^{r_0}$ of $X$ by disjoint open balls of $Y$ of radius at most $r_0$ as follows. Let $\{x_0, x_1, \ldots \}$ be an enumeration of $X$, and choose $r<r_0$ \st\ the boundary of the ball $\cu_0= B_r(x_0)$ is disjoint from $X$. This is possible because $X$ is countable and there are uncountably many values of $r$ to  choose from. Proceed inductively, letting $\cu_i, i=1,2,\ldots$ be a ball, of radius less than $r_0$, around the next still uncovered element of $\{x_1, x_2, \ldots \}$, \st\ $\cu_i$ is disjoint from all $\cu_j, j<i$, and $\partial \cu_i \cap X=\emptyset$. 

This completes the definition of a cover $\cu^{r_0}$ of $X$. We proceed to define a sequence of such covers $\cu^{r_i}$, where $r_i \to 0$ inductively: in each step $i$, we repeat the above construction, with $X$ replaced by its intersection with  each element $U$ of $\cu^{r_{i-1}}$, to obtain a refinement $\cu^{r_{i}}$ of $\cu^{r_{i-1}}$. Nestedness is easily preserved by our inductive step, and so the union $\bigcup_{i\in \N} \cu^{r_i}$ yields a \nted\ basis of $X$ as desired.
\end{proof}

Fix some $n>1 \in \N$ for the rest of this paper. A \defi{(closed) $k$-punctured sphere} in $\BS^n$ is a connected sub-manifold with $k\in \N_{>0}$ boundary components, each of which is  \pfl\  and homeomorphic to  $\BS^{n-1}$. By a \defi{\pfl} subspace here we mean a finite union of open metric balls. We will prove \Tr{main thm} by decomposing \smA\ (and \smZ) as a union of punctured spheres with non-intersecting interiors, and applying the following basic fact to each of them.

\begin{proposition} \label{hom pun}
Let $M,M'$ be two $k$-punctured spheres in $\BS^n$, and $f$ a homeomorphism from one of the boundary components of $M$ onto one of the boundary components of $M'$. Then there is a homeomorphism $f'$ from $M$ onto $M'$ extending $f$. Moreover, $f'$ preserves orientation \iff\ $f$ does. 
\end{proposition}
\begin{proof}[Proof (sketch)]
Continuously deform each of $M,M'$ to a \emph{standard}  $k$-punctured sphere, i.e.\ one with its boundary components having specific sizes and shapes, and lying in specific locations. Then compose one of these deformations with the inverse of the other.
\end{proof}

The condition of \pfl ness that we imposed on the boundary components of $M, M'$ is probably stronger than needed, but some condition is necessary to avoid obstructions such as Alexander's horned sphere, a subset of $\BS^3$  homeomorphic to $\BS^2$, the two sides of which are not homeomorphic to each other \cite{Alexander}, \cite[p.~169]{Hatcher}.

\medskip
An \defi{automorphism} of a topological space $Y$ is a homeomorphism from $Y$ onto itself.

\section{Proofs of main results}

\Tr{main thm} is an immediate consequence of the following statement, which replaces countability by the weaker ---by \Lr{lem nted}--- condition of having \nted\ bases consisting of open balls.

\begin{theorem} \label{main thm unct}
Let $A,Z$ be two compact subsets of $\BS^n, n\geq 2$ that admit \nted\ bases $\cb^A,\cb^Z$ consisting of open metric balls of $\BS^n$, and let $h: A \to Z$ be a homeomorphism. Then there is an automorphism $h'$ of $\mathbb{S}^n$ extending $h$. 

Moreover, $h'$ can be chosen to be orientation reversing (or preserving).
\end{theorem}
\begin{proof}
We will construct $h'$ inductively, in infinitely many steps, with each step extending the previous one by mapping punctured spheres of \smA\ to punctured spheres of \smZ\ via \Lr{hom pun}. The punctures will have radii converging to 0, which will ensure that eventually all of \smA\ is mapped onto \smZ.

To start our inductive process, we choose two closed metric balls $M_0,M_0'$ in $\BS^n$ that avoid $A,Z$, respectively, and a homeomorphism $h_0: M_0 \to M_0'$. In fact, since $A\cup Z$ is closed, we can let $M_0 = M_0'$ be a small enough ball of any point outside $A\cup Z$, and we can let $h_0$ be the identity or any orientation reversing isometry of $M_0$. We make the latter choice depending on whether we want the homeomorphism $h'$ in the statement to be orientation preserving or reversing.

For the inductive step $n=1,2,\ldots$, assume that we have already defined a homeomorphism $h_{n-1}: M_{n-1} \to M'_{n-1}$ between punctured spheres of $\BS^n$ with the following properties. For a punctured sphere $M$, we let $\cb\cc (M)$ denote the set of boundary components of $M$.
\begin{enumerate}
\item \label{h i} $M_{n-1}$ 
  is a punctured sphere in \smA; 
is the boundary of a ball in $\cb^A$; 
\item \label{h iii} each $S \in \cb\cc (M_{n-1})$ has radius at most $1/n$ (any sequence converging to 0 would do).

Given $S \in \cb\cc (M_{n-1})$, notice that exactly one of the two components of $\BS^n \sm S$ meets $M_{n-1}$; we let $\ocirc{S}$ denote the other component (which is a 1-punctured sphere). Notice that $h_{n-1}(S)$ must coincide with one of the boundary components of $M'_{n-1}$. Thus $h_{n-1}$ induces a bijection $S \mapsto S'$ from $\cb\cc (M_{n-1})$ to $\cb\cc (M'_{n-1})$. Our next property says that this bijection respects $h$:

\item \label{h iv} \fe\ $S  \in \cb\cc (M_{n-1})$, we have $h(\ocirc{S} \cap A) = \ocirc{S'} \cap Z$
\end{enumerate}

Moreover, we assume ---and ensure inductively--- that properties \ref{h i}--\ref{h iv} hold with the roles of $A, M_{n-1}, h_{n-1}$  exchanged with $Z, M'_{n-1}, h^{-1}_{n-1}$, respectively. 

Notice that these properties are satisfied for $n=1$. In particular, there is a unique $S \in \cb\cc (M_{n-1})$, and we have $\ocirc{S} \cap A = A$ and $\ocirc{S'} \cap Z = Z$.

Assuming that such an $h_{n-1}$ has already been defined, we proceed to define $M_n$ and $h_n$ as follows. If $n$ is odd, \fe\ $S  \in \cb\cc (M_{n-1})$, let $C_S$ be a cover of $\ocirc{S} \cap A$ by elements of $\cb^A$ of radius  at most $1/(n+2)$ that are contained in $\ocirc{S}$. Since $A$, and hence $\ocirc{S} \cap A$, is compact, we may assume that $C_S$ is finite. Moreover, its elements can be chosen to be  pairwise disjoint since  $\cb^A$ is \nted. Let $M_S:= \ocirc{S} \sm \bigcup C_S$, and notice that $M_S$ is a $k$-punctured sphere with $k:= 1+|C_S|$; indeed, the boundary components of $M_S$ are $\partial S$ and the boundaries of the balls in $C_S$.

We are looking for an appropriate $k$-punctured sphere $M'_S$ in $\smZ$ to map $M_S$ onto. To find it, for every $U\in C_S$, let $C_U$ be a finite cover of $h(\ocirc{U} \cap A)$ by pairwise disjoint elements of $\cb^Z$, chosen small enough that \ref{h iii} is satisfied and $\bigcup C_U \subset \ocirc{S'}$ holds; the latter is possible  by \ref{h iv}. Notice that $\ocirc{S'} \sm \bigcup_{U\in C_S} C_U$ is a $k'$-punctured sphere, but we cannot yet use it as the desired $M'_S$ because $k'$ will typically be larger than $k$. To amend this, we combine, for each $U\in C_S$, the balls in $C_U$ into one topological ball $S_U$ as follows.  First, we join the elements of $C_U$ by \pl\ arcs in $\ocirc{S'}$, so as to form a connected and simply connected set (we have to use exactly $|C_U|-1$ arcs for this). Then, we blow up each of these arcs into a sufficiently small tube consisting of open metric balls. It is not hard to do this so that the resulting union $S_U$ of  $\bigcup C_U$ and all these tubes is homeomorphic to a ball in $\BS^n$, and so that $S_U$ is disjoint from $S_V$ for $U\neq V \in C_S$ (\fig{figSU}). Moreover, it is easy to ensure that $\partial S_U$ is \pfl. Here, we used the fact that the elements of $C_U$ are pairwise disjoint.
We can now let $M'_S:= \ocirc{S'} \sm \bigcup_{U\in C_S} S_U$, which is a $k$-punctured sphere as desired.

\begin{figure} 
\begin{center}
\begin{overpic}[width=0.8\linewidth]{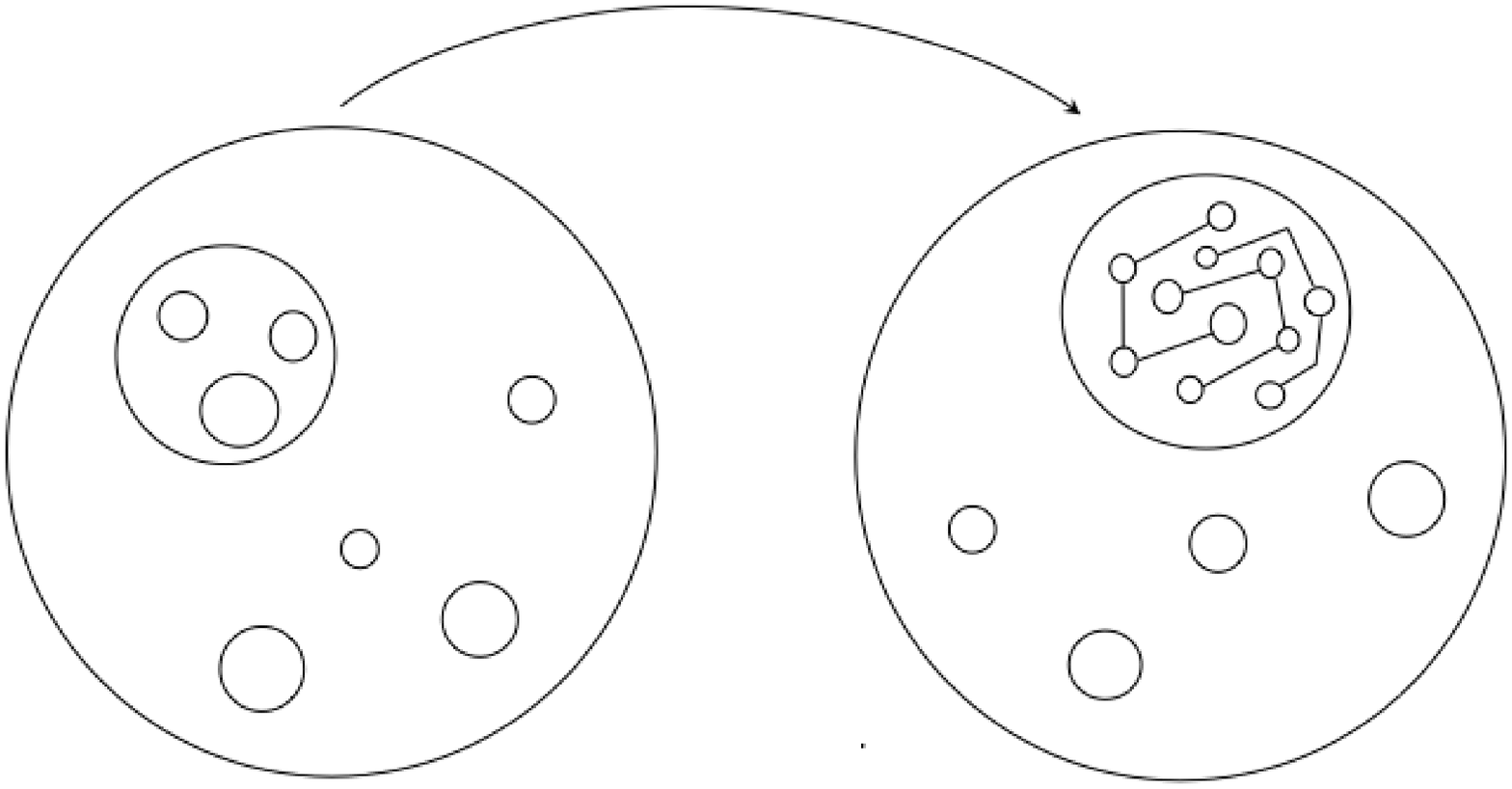} 
\put(22,34){$S$}
\put(65,35){$S'$}
\put(1,43){$M_{n-1}$}
\put(45,55){$h_{n-1}$}
\put(15,25){$U$}
\end{overpic}
\end{center}
\caption{Mapping the punctured sphere $M_S$ onto a punctured sphere $M'_S$ inside $S'$. Each element $U$ of $C_S$ (there are 3 in this example picture) gives rise to a topological ball $S_U$, obtained by joining metric balls with tubes.}
\label{figSU}
\end{figure}

Let $M_n$ be the union of $M_{n-1}$ with all these punctured spheres of the form $M_S$ for $S\in \cb \cc(M_{n-1})$. Recall that $h_{n-1}$ restricts to a homeomorphism from $S$ onto $S'$, which is one of the punctures of $M'_S$. Applying \Prr{hom pun} to each $S\in \cb \cc(M_{n-1})$, we thus obtain an extension $h_n$ of $h_{n-1}$ to  $M_n$, with image $M'_n:= M'_{n-1} \cup \bigcup_{S\in \cb \cc(M_{n-1})} M'_S$. Moreover, 
\labtequ{orient}{$h_n$ will be orientation preserving \iff\ $h_{n-1}$ is,}
because this property is witnessed by the restriction of $h_{n-1}$ to any $S\in \cb \cc(M_{n-1})$.

This completes the definition of $h_n$ for odd $n$. 
For even $n$ we repeat the construction verbatim, except that we exchange the roles of $A$ and $Z$; this will define homeomorphisms from  punctured spheres of $\BS^n \sm (Z \cup  M'_{n-1})$ to punctured spheres of $\BS^n \sm (A \cup  M_{n-1})$, and we use the inverses of these homeomorphisms to extend $h_{n-1}$ into $h_n$.
(The reason for treating odd and even $n$ differently is to make sure that the aforementioned tubes do not cause a problem by converging to points in \smA, which points would then fail to be in the image of $\bigcup_{\nin} h_n$.)

Note that \ref{h i} is satisfied, because $M_n$ and $M'_n$ have been obtained by glueing internally disjoint punctured spheres along a common boundary component. For $n$ odd, $M_n$ satisfies \ref{h iii} because each of  its boundary components 
coincides, by construction, with the boundary $\partial U$ of some element of $\cb^A$ of radius  at most $1/(n+2)$, which is less than the desired $1/(n+1)$. The boundary components of $M'_n$ can be bigger because of the tubes, but still they inherit the  desired bound from the previous (even) step. For even $n$ the same argument with the roles interchanged applies.
To see that \ref{h iv} is satisfied, notice that $\cb \cc(M_{n})$ consists of spheres $U\in C_S$ as in the above construction, and we have $h_n(\partial{U})=\partial{S_U}$; in other words, we have $\partial{U}' = \partial{S_U}$. Moreover, we have  $\ocirc{S_U}\cap Z = \bigcup C_U= h(\ocirc{U}\cap A)$ as requirered by \ref{h iv}. 

\medskip
Having defined $h_n$ for all $n$, we let $h':= h \cup \bigcup_{\nin} h_n$. 

Notice that every point $p\in \smA$ is eventually in $M_n$, hence in the domain of $h'$, because of \ref{h iii}, the fact that $d(p,A)>0$ as $A$ is closed, and the fact that  $p\in \ocirc{S}$ for some $S\in \cb \cc(M_{n})$ \fe\ $n$ by \ref{h i}. By the same argument, $h'$ surjects onto $\BS^n$.

Easily, $\bigcup_{\nin} h_n$ is a homeomorphism, because it is one locally by the construction of the $h_n$. It is straightforward to check that $h'$ is a homeomorphism at any $p\in \smA$ too, because arbitrarily small elements $U$ of $\cb^A$ containing $p$ are mapped to $h(U)$ by \ref{h iv}.

The final statement about orientation follows from \eqref{orient} and the fact that $h_0$ gives us the desired choice, as mentioned above.

\end{proof}

\begin{remark}
In \Tr{main thm}, we can construct $h'$ to be \pl\ (or smooth), because all objects involved in its construction can be chosen to be \pl.
\end{remark}


Recall that a Cantor set $C \subset \mathbb{S}^n$ is called \defi{tame}, if $\mathbb{S}^n \sm C$ is homeomorphic to the complement of a standard Cantor set in  $\mathbb{S}^n$. It is well-known that $C$ is tame \iff\ there is an automorphism $h: \mathbb{S}^n \to \mathbb{S}^n$ such that $h(C)$ lies in a piecewise-linear arc \cite{BinTam}. Easily, $h(C)$ admits a \nted\ basis in this case, and so \Tr{main thm unct} provides another proof of the well-known fact that every tame Cantor set is $C$ strongly homogeneously embedded \cite[p.~93]{MoiseGT}\footnote{Moise \cite{MoiseGT} proves this in 2-dimensions, but the higher dimensional case follows easily by induction, because we may assume that $C$ is contained in the equator of $S^n$ using the aforementioned  automorphism $h$.}. 

It follows from \Tr{main thm} that for every countable, compact, subset $Z$ of $\BS^n, n\geq 2$, \ti\ an automorphism $h$ of $\BS^n$  \st\ $h(Z)$ lies in a linear arc. This is because every such space $Z$ ---more generally, every zero-dimensional separable metric space--- is homeomorphic with a subset of a real interval \cite[1.3.17]{EngelkingDT}.


\section{A converse}

In the converse direction of \Tr{main thm unct}, we remark that if $A, Z$ are two compact, totally disconnected subsets of $\BS^n$ \st\ $\smA$ is homeomorphic to $\smZ$, then $A$ is homeomorphic to $Z$. One way to see this is to use the fact that homeomorphic spaces have homeomorphic Freudenthal compactifications, combined with

\begin{proposition} \label{prof Freu}
Let $A$ be a compact and totally disconnected subspace of $\BS^n$. Then the identity map on $\smA$ extends to a homeomorphism between the Freudenthal compactification of \smA\ and $\BS^n$.
\end{proposition}
In particular, $A$ is homeomorphic to the space of ends of \smA. Proving \Prr{prof Freu} is easy when $A$ admits a \nted\ basis of open metric balls of $\BS^n$, and slightly more difficult in general. 

\comment{
\begin{proof}[Proof of \Prr{prof Freu} (sketch)]
Given $x,y\in A$, we will find a compact and connected subspace of \smA\ separating $x$ from $y$ in $\BS^n$. The rest of the proof is straightforward.

Since $A$ is totally disconnected, there are disjoint open subsets $U \ni x, V\ni y$ of $A$ \st\ $A=U\cup V$. Let $\epsilon:= d(U,V)$. Let $C$ be a cover of $A$ by 
\end{proof}
}

\section{Final remarks}
A topological space homeomorphic to the complement of a tame Cantor set in $\BS^3$ is called a \defi{\CtS}. We remark that this space is important, as it is the only \sico, infinitely-ended, open 3-manifold that can cover a closed 3-manifold; see \cite{GeoKon} for details.

The interested reader will find many interesting problems on wild Cantor sets in \cite{GarRepCan}.

\medskip
In this paper we worked with $\BS^n$ as the host space. 
How far can we extend \Tr{main thm} beyond spheres?
\begin{problem}
For which $n$ is it true that for every $n$-manifold $M$, and every two countable, compact, homeomorphic subsets $A,Z$ of $M$, the complements $M\sm A, M\sm Z$ are homeomorphic?
\end{problem}

\Tr{main thm} tells us that we can extend each homeomorphism $h: A \to Z$ to a homeomorphism between their host spaces. It could be interesting to try to extend several such homeomorphisms simultaneously, so that they form a specific group. We propose two problems of this kind:
\begin{problem}
Let $Z$ be a countable subset of $\R^n,n\geq 2$ with no accumulation points in $\R^n$. Which groups of bijections of $Z$ extend to groups of automorphisms of $\R^n$?
\end{problem}

For example, suppose $n=3$ and $\pi$ is a bijection of $Z$ that has a finite cycle $(z_1,z_2, \ldots, z_k)$ and fixes all other elements of $Z$ pointwise. Let $\Gamma$ be the finite cyclic group generated by $\pi$. Then it is possible to realise $\Gamma$ as a group of homeomorphisms of $\R^3$ as follows. Find a 2-way infinite arc $A$ that contains all of $Z$ except $\{z_1,z_2, \ldots, z_k\}$, and find a rotation $h$ of $\R^3$ around $A$ that maps each $z_i$ to $z_{i+1 \pmod k}$. Then $h$ generates a finite cyclic group that acts on $Z$ the same way as  $\Gamma$.

\begin{problem}
Is there $\nin$ \st\ for every  countable, closed subset  $Z$  of $\BS^n$, every group of automorphisms of $Z$  extends into a group of automorphisms of $\BS^n$?
\end{problem}

Call a subset  $Z$ of $\BS^n$ \defi{ambiently-reversible}, if there is an orientation-reversing automorphism $h: \BS^n \to \BS^n$ fixing $Z$ pointwise. Is there an ambiently-reversible wild Cantor set in $\BS^n, n \geq 3$? Are all Cantor sets in $\BS^n$ ambiently-reversible? More generally, we have 

\begin{problem}
Which subsets of $\BS^n, n \geq 2$ are ambiently-reversible?
\end{problem}

\bibliographystyle{plain}
\bibliography{../collective}

\end{document}